\documentclass[11pt,reqno]{amsart}
\usepackage[T1]{fontenc}
\usepackage[mathscr]{eucal}
\usepackage{amsmath}
\usepackage{amssymb}
\usepackage{amsthm}
\usepackage{amsfonts}
\usepackage{xcolor}
\usepackage{graphicx}
\usepackage[shortlabels]{enumitem}
\usepackage{bm}
\usepackage{float}
\usepackage[left=1.5in,right=1.5in,top=2in,bottom=2in]{geometry}
\usepackage{booktabs}
\usepackage{makecell}
\usepackage[hidelinks]{hyperref}

\renewcommand{\S}{\mathbb S}
\renewcommand{\H}{\mathbb H}
\newcommand{\CP}{\mathbb C\mathrm P}
\newcommand{\CH}{\mathbb C\mathrm H}
\newcommand{\Z}{\mathbb Z}
\newcommand{\R}{\mathbb R}

\newcommand{\set}[1]{ \left\{ #1 \right\} }

\newcommand{\sol}{\mathrm{Sol}^4_0}

\newcommand{\Id}{\mathrm{Id}}

\newcommand{\SO}{\mathrm{SO}}
\renewcommand{\O}{\mathrm{O}}
\newcommand{\U}{\mathrm U}

\newcommand{\SL}{\mathrm{SL}}
\renewcommand{\phi}{\varphi}

\theoremstyle{plain}
\newtheorem{theorem}{Theorem}[section]
\newtheorem{proposition}[theorem]{Proposition}
\newtheorem{corollary}[theorem]{Corollary}

\newtheorem{definition}[theorem]{Definition}
\newtheorem{remark}[theorem]{Remark}

\numberwithin{equation}{section}

\title[Parallel and totally umbilical hypersurfaces of Sol$_0^4$]{Parallel and totally umbilical hypersurfaces of the four-dimensional Thurston geometry $\text{Sol}^4_0$}

\author{Marie D'haene \and Jun-ichi Inoguchi \and Joeri Van der Veken}

\address{M. D'haene and J. Van der Veken, KU Leuven, Department of Mathematics, Celestijnenlaan 200 B -- Box 2400, 3001 Leuven, Belgium}
\email{marie.dhaene@kuleuven.be}
\email{joeri.vanderveken@kuleuven.be}

\address{J. Inoguchi,
Department of Mathematics,
Hokkaido University,
Sapporo 060-0810,
Japan}
\email{inoguchi@math.sci.hokudai.ac.jp}

\thanks{
M.\ D'haene is supported by the Methusalem grant METH/15/026 of the Flemish Government. \\
J.\ Inoguchi is partially supported by JSPS KAKENHI Grant Number 19K03461. \\
J.\ Van der Veken is supported by the Research Foundation--Flanders (FWO) and the National Natural Science Foundation of China (NSFC) under collaboration project G0F2319N, by the KU Leuven Research Fund under project 3E210539 and by the Research Foundation--Flanders (FWO) and the Fonds de la Recherche Scientifique (FNRS) under EOS Projects G0H4518N and G0I2222N}

\date{}
\keywords{
Thurston geometry, solvable Lie group,
Codazzi hypersurface, totally geodesic, totally umbilical, parallel second fundamental form}

\makeatletter
\@namedef{subjclassname@2020}{\textup{2020} Mathematics Subject Classification}
\makeatother
\subjclass[2020]{Primary 53C42; Secondary 53C30}

\begin{document}

\begin{abstract}
We study hypersurfaces of the four-dimensional Thurston geometry $\sol$, which is a Riemannian homogeneous space and a solvable Lie group.
In particular, we give a full classification of hypersurfaces whose second fundamental form is a Codazzi tensor --including totally geodesic hypersurfaces and hypersurfaces with parallel second fundamental form-- and of totally umbilical hypersurfaces of $\sol$.
We also give a closed expression for the Riemann curvature tensor of $\sol$, using two integrable complex structures.
\end{abstract}

\maketitle

Located on the crossroads of geometry, algebra and group theory, Riemannian homogeneous spaces are a fundamental class of manifolds to study.
Thurston geometries form an interesting subset of Riemannian homogeneous spaces, which have been studied extensively in dimension three due to Thurston's geometrization for three-manifolds.
This is a deep result listing the possible Riemannian structures of compact orientable three-manifolds, very similar to the uniformization theorem for compact orientable surfaces.
Roughly stated, any three-manifold can be decomposed into pieces, each of which admits a Riemannian metric locally isometric to one of eight three-dimensional model spaces, the Thurston geometries $\R^3$, $\S^3$, $\H^3$, $\S^2\times\R$, $\H^2\times\R$, $\widetilde{\SL}(2,\R)$, $\text{Nil}^3$ and $\text{Sol}^3$.
Up until today, no such geometrization is known in dimension four.
However, the study of four-dimensional Thurston geometries, and in particular of their submanifolds, is interesting on its own from a Riemannian geometric point of view.
Moreover it is an attempt to shed light on the world of four-manifolds using low-dimensional techniques.
There are 19 types of Thurston geometries in dimension four.
We have listed them in Table~\ref{tab:geometries}, ordered by the stabilizer of the action of the \emph{connected component of the identity} of their isometry groups.
This list is based on \cite{Filip} and \cite{Wall}.
\begin{table}[h]
	\centering
	\begin{tabular}{c c}
		\toprule
		Thurston geometry & Stabilizer \\
		\midrule
		$\R^4,\; \S^4,\; \H^4$
		& $\SO(4)$ \\[1ex]
		$\CP^2,\; \CH^2$
		& $\U(2)$ \\[1ex]
		$\S^3 \times \R,\; \H^3 \times \R$
		& $\SO(3)$ \\[1ex]
		\makecell{$\S^2 \times \R^2,\; \S^2 \times \S^2$ \\ $\H^2 \times \S^2,\; \H^2 \times \R^2,\; \H^2 \times \H^2$}
		& $\SO(2) \times \SO(2)$ \\[2.5ex]
		$\sol,\; \widetilde{\SL}(2,\R) \times \R,\; \text{Nil}^3 \times \R$
		& $\SO(2)$ \\[1ex]
		$\text{F}^4$ & $(\S^1)_{1,2}$ \\[1ex]
		$\text{Nil}^4,\; \text{Sol}^4_1,\; \text{Sol}^4_{m,n}$
		& $\set{1}$ \\
		\bottomrule \\
	\end{tabular}
	\caption{Four-dimensional Thurston geometries.}
	\label{tab:geometries}
\end{table}

The main goal of the present paper is to classify some fundamental families of hypersurfaces of $\sol$.
Theorem~\ref{thm:codazzi} gives a complete classification of hypersurfaces of $\sol$ for which the second fundamental form $h$ is a Codazzi tensor, also called \emph{Cod\-azzi hypersurfaces}.
They are defined by the property that the covariant derivative $\nabla h$ is totally symmetric, where $\nabla$ is the Levi Civita connection of the induced metric.
Obviously, the class of Codazzi hypersurfaces contains the class of hypersurfaces with parallel second fundamental form, also called \emph{parallel hypersurfaces}, for which $\nabla h = 0$.
This class, in its turn, contains the class of \emph{totally geodesic hypersurfaces}, for which $h=0$.
The latter two families of hypersurfaces of $\sol$ are classified in Corollary~\ref{cor:parallel} and Corollary~\ref{cor:tot-geod} respectively and have interesting geometric interpretations.
Any geodesic of a totally geodesic submanifold is also a geodesic of the ambient space, whereas parallel submanifolds can be thought of as submanifolds which `look the same everywhere from an extrinsic viewpoint', since the second fundamental form contains all the extrinsic geometric information about a submanifold.
We also consider another class containing the totally geodesic hypersurfaces, namely \emph{totally umbilical hypersurfaces}.
They are defined by the property $h(X,Y)=g(X,Y)H$ for all tangent vector fields $X$ and $Y$, where $g$ is the induced metric and $H$ is a fixed normal vector field, which is necessarily the mean curvature vector field.
Alternatively, all shape operators of a totally umbilical hypersurface are multiples of the identity transformation, which means that, at any of its points, such a hypersurface `looks the same in any direction'.
A full classification of totally umbilical hypersurfaces of $\sol$ can be found in Theorem~\ref{thm:totally-umbilical}.

These results fit in the rich history of research on special submanifolds of Thurston spaces.
We mention some relevant results below, without aiming to provide a complete list.
The classification of totally geodesic and parallel submanifolds of the real space forms $\R^n$, $\S^n$ and $\H^n$ goes back to independent work in~\cite{Backes-Reckziegel} and~\cite{Takeuchi}, whereas totally geodesic and parallel hypersurfaces of real space forms were already classified in~\cite{Lawson}.
A study of parallel submanifolds of the complex space forms $\CP^n$ and $\CH^n$ can be found in~\cite{naitoh1983parallel} and a classification of parallel and totally umbilical hypersurfaces of $\S^n \times \R$ and $\H^n \times \R$ in \cite{VV} and \cite{CKV}.
In dimension three, a classification of totally umbilical surfaces in $\S^2 \times \R$ and $\H^2 \times \R$,
together with the remark that such surfaces do not exist in $\widetilde{\SL}(2,\R)$ and $\text{Nil}^3$, was independently obtained in \cite{toub-souam} and \cite{vdv-bcv}.
The former paper also contains the classification of totally umbilical surfaces in $\text{Sol}^3$.
In this context, we should also mention the fundamental existence and uniqueness theorem for isometric immersions of surfaces into three-dimensional homogeneous spaces with four-dimensional isometry groups, which are essentially the Thurston spaces $\S^2\times\R$, $\H^2\times\R$, $\widetilde{\SL}(2,\R)$, $\text{Nil}^3$ and Berger spheres, obtained in \cite{daniel2007isometric}.
Parallel and totally geodesic surfaces in these spaces were classified in \cite{Belkhelfa-Dillen-Inoguchi}, while the classification in all remaining three-dimensional homogeneous spaces, including $\text{Sol}^3$, can be found in \cite{I-VdV-2007} and \cite{I-VdV-2008}.
Finally, it was shown in \cite{deLVdV} that there are no Codazzi hypersurfaces in the four-dimensional Thurston geometry $\text F^4$,
while the recent preprint \cite{nil4} contains a classification of Codazzi hypersurfaces of $\text{Nil}^4$.

The paper is organized as follows: in Section~\ref{sec:prelim} we briefly recall some definitions and fundamental formulas concerning hypersurfaces.
In Section~\ref{sec:sol40} we describe $\sol$ as a Lie group, a Thurston geometry and a Riemannian manifold.
In particular, we give an invariant expression for its curvature tensor in terms of two complex structures and a projection operator (formula \eqref{eq:R}).
Finally, in Sections~\ref{sec:codazzi} and~\ref{sec:tot-umb} we classify Codazzi hypersurfaces and totally umbilical hypersurfaces of $\sol$, respectively.
The main results are Theorem~\ref{thm:codazzi}, Corollary~\ref{cor:parallel}, Corollary~\ref{cor:tot-geod} and Theorem~\ref{thm:totally-umbilical}.

\section{Preliminaries}
\label{sec:prelim}
Let $(\tilde{M},\tilde g)$ be a Riemannian manifold with Levi-Civita connection $\tilde \nabla$ and consider an isometrically immersed hypersurface $(M,g)$ of $(\tilde{M},\tilde g)$. The Gauss formula relates the Levi-Civita connection $\nabla$ of the hypersurface to the one of the ambient space:
\begin{equation}
	\label{eq:gauss-formula}
\tilde\nabla_{X}Y=\nabla_{X}Y + h(X,Y)
\end{equation}
for $X,Y \in \mathfrak X(M)$.
Here, $h$ is the {second fundamental form}, a symmetric tensor which takes values in the normal bundle of $M$.
Any normal vector field $\xi$ induces a self-adjoint endomorphism $A_\xi: \mathfrak X(M) \to \mathfrak X(M)$, called the {shape operator} associated to~$\xi$, via the Weingarten formula:
\begin{equation}
	\label{eq:weingarten-formula}
\tilde\nabla_{X}\xi = -A_\xi X + \nabla^\perp_X \xi
\end{equation}
for $X \in \mathfrak X(M)$.
Here, $\nabla^\perp$ is the normal connection.
From the definition, one can see that $A_\xi$ is metrically equivalent to $h$, meaning $\tilde g(h(\cdot\,,\cdot),\xi) = g(A_\xi \, \cdot\,,\cdot)$. If $N$ is a local \emph{unit} normal vector field, then the Weingarten formula reduces to
\[
	\tilde \nabla_X N = -A_N X.
\]
The {mean curvature vector field} $H$ is defined as $H=\tfrac1m\text{trace}_{g}(h)$ where $m$ is the dimension of $M$.
If $\set{E_1,\ldots,E_m}$ is a local orthonormal frame on $M$, then $H$ can be computed by
\[
	H = \frac1m \sum_{i=1}^m h(E_i,E_i).
\]

In Sections \ref{sec:codazzi} and \ref{sec:tot-umb} we consider special types of hypersurfaces of $\sol$. We collect their definitions below. Note that  the covariant derivative of $h$ is defined as
\[
	(\nabla h)(X,Y,Z) = \nabla^\perp_X h(Y,Z) - h(\nabla_XY,Z) - h(Y,\nabla_XZ)
\]
for vector fields $X$, $Y$ and $Z$ tangent to the hypersurface.

\begin{definition}
	\label{def:submanifolds}
Using the notations introduced above, we say that a hypersurface is
\begin{enumerate}[(a)]
	\item totally geodesic if $h=0$,
	\item parallel if $\nabla h = 0$,
	\item Codazzi if $\nabla h$ is symmetric in its three components,
	\item totally umbilical if $h(\cdot\,,\cdot)=g(\cdot\,,\cdot)H$.
\end{enumerate}
\end{definition}
Note that, since $h$ is symmetric, if suffices to require that $\nabla h$ is symmetric in the first two components for a hypersurface to be Codazzi.
Also, if $N$ is a unit normal vector field to a hypersurface and we set $\lambda = \tilde g(H,N)$, the condition for being totally umbilical is equivalent to $A_N = \lambda\,\Id$.

The {equations of Gauss and Codazzi} relate the Riemann curvature tensor $\tilde R$ of $(\tilde M, \tilde g)$ to the Riemann curvature tensor $R$ of $(M,g)$:
\begin{align*}
& (\tilde R(X,Y)Z)^t  =
  R(X,Y)Z - A_{h(Y,Z)}X + A_{h(X,Z)}Y, \tag{Gauss} \\
  & (\tilde R(X,Y)Z)^n  =
  (\nabla h)(X,Y,Z) - (\nabla h)(Y,X,Z) \tag{Codazzi}
\end{align*}
for $X,Y,Z \in \mathfrak X(M)$, where the superscripts $t$ and $n$ denote the tangent and normal components of a vector field along $M$.

\section{The geometry of $\sol$}
\label{sec:sol40}
In this section we give a description of $\sol$, partly based on~\cite{j-traj}.

\subsection*{The Lie group Sol\raisebox{0.1em}{\textsuperscript{4}}\kern-5pt\raisebox{-0.1em}{\textsubscript{0}}}
The underlying manifold of the Thurston geometry $\sol$ is the connected solvable Lie group consisting of matrices of the form
\[
  M(x,y,z,t) =
	\left(
  \begin{array}{cccc}
  e^{t} & 0 & 0 &x\\
  0 & e^{t} &0 & y\\
  0 & 0 & e^{-2t} & z\\
  0 & 0 & 0 &1
  \end{array}
  \right),
\]
where $x,y,z,t \in \R$.
Equivalently, $\sol$ is the Lie group $(\R^{4},\cdot\,)$ with group operation
\begin{equation}
\label{eq:groupoperation}
(a,b,c,d)\cdot
(x,y,z,t)
=(a+e^{d}x,\, b+e^{d}y,\, c+e^{-2d}z,\, d+t).
\end{equation}
We use the identification of $M(x,y,z,t)$ with the tuple $(x,y,z,t)$ throughout the rest of the paper.
The inverse of $(x,y,z,t)$ is given by
\[
(x,y,z,t)^{-1}=(-e^{-t}x,-e^{-t}y,-e^{2t}z,-t).
\]

The Lie algebra of $\sol$ is spanned by the basis $\set{e_1,e_2,e_3,e_4}$ given by
\begin{alignat*}{2}
	e_1 &=
	\left(
	\begin{array}{cccc}
	0 & 0 & 0  &1\\
	0 & 0 &  0 & 0\\
	0 & 0 & 0 & 0\\
	0 & 0 & 0 & 0
	\end{array}
	\right), \qquad
	e_2 &&=
	\left(
	\begin{array}{cccc}
	 0 & 0 & 0  &0\\
	 0 & 0 &  0 & 1\\
	 0 & 0 & 0 & 0\\
	 0 & 0 & 0 &0
	\end{array}
	\right), \\[1ex]
	e_3 &=
	\left(
	\begin{array}{cccc}
	 0 & 0 & 0  &0\\
	 0 & 0 &  0 & 0\\
	 0 & 0 & 0 & 1\\
	 0 & 0 & 0 &0
	\end{array}
	\right), \qquad
	e_4 &&=
	\left(
	\begin{array}{cccc}
	1 & 0 & 0 & 0  \\
	0 & 1 & 0 &  0 \\
	0 & 0 & -2 & 0 \\
	 0 & 0 & 0 &0
	\end{array}
	\right).
\end{alignat*}
The left invariant vector fields determined by
$e_1$, $e_2$, $e_3$ and $e_4$ are
\begin{equation}
	\label{eq:onf}
	E_1=e^{t}\frac{\partial}{\partial x},
	\quad
	E_2=e^{t}\frac{\partial}{\partial y},
	\quad
	E_3=e^{-2t}\frac{\partial}{\partial z},
	\quad
	E_4=\frac{\partial}{\partial t}.
\end{equation}
These vector fields satisfy the commutation relations
\begin{equation} \label{eq:Lie_brackets}
\begin{aligned}
& [E_1 , E_2 ] = [E_1 , E_3 ] = [E_2 , E_3 ] = 0, \\
& [E_1 , E_4 ] = -E_1, \quad
  [E_2 , E_4 ] = -E_2, \quad
  [E_3 , E_4 ] = 2E_3.
\end{aligned}
\end{equation}

\subsection*{The Thurston geometry Sol\raisebox{0.1em}{\textsuperscript{4}}\kern-5pt\raisebox{-0.1em}{\textsubscript{0}}}
A Thurston geometry is a pair of a manifold and a Lie group acting on it.
The Lie group that acts on $\sol$ making it a Thurston geometry is the semi-direct product $\sol \rtimes \left( \O(2) \times \Z_2 \right)$ (see for example~\cite{hillman2002four}).
Here, $\sol$ acts on itself by left translations, $\O(2)$ acts by rotations and reflections in the $xy$-plane and $\Z_2$ acts by reflections in the $z$-coordinate.
Hence the stabilizer of the identity is $\O(2) \times \Z_2$.
Note that the stabilizers in Table~\ref{tab:geometries} are those of the actions of the \emph{connected components of the identity} of the groups.
In the case of $\sol$, that stabilizer is $\SO(2)$.

\subsection*{The Riemannian manifold Sol\raisebox{0.1em}{\textsuperscript{4}}\kern-5pt\raisebox{-0.1em}{\textsubscript{0}}}
Since $\sol$ is a reductive homogeneous space, there is a bijective correspondence between Riemannian metrics invariant under the action of $\sol \rtimes \left( \O(2) \times \Z_2 \right)$ on the one hand, and inner products on the Lie algebra of $\sol$ invariant under the adjoint action of the stabilizer, $\O(2) \times \Z_2$, on the other hand.
By determining these inner products and left translating them to the whole manifold, one finds that any invariant Riemannian metric is of the form
\begin{equation*}
	\tau_1 e^{-2t}(dx^2 + dy^2) + \tau_2 e^{4t} dz^2 + \tau_3 dt^2,
\end{equation*}
where $\tau_1,\tau_2,\tau_3 > 0$.
Via a homothety we may set $\tau_3 = 1$. Moreover, by rescaling the $x$-, $y$- and $z$-coordinates we can also eliminate the parameters $\tau_1$ and $\tau_2$. Hence, there is only one metric invariant under the action of $\sol \rtimes (\O(2) \times \Z_2)$ up to homotheties, namely
\begin{equation}
	\label{eq:metric}
	\tilde g = e^{-2t}(dx^2 + dy^2) + e^{4t} dz^2 +  dt^2.
\end{equation}
From now on, we will always consider this metric $\tilde g$ on $\sol$. Note that the frame $\{E_1,E_2,E_3,E_4\}$ from \eqref{eq:onf} is then orthonormal.

The complex structures $J_+$ and $J_-$, defined by the following equalities, are integrable and compatible with~$\tilde g$:
\begin{equation}
	\label{eq:complex-structures}
	\!\begin{aligned}
		J_+E_1 = E_2,&\quad J_+E_2 = -E_1,& &J_+E_3 = E_4,& &J_+E_4 = -E_3, \\
		J_-E_1 = E_2,&\quad J_-E_2 = -E_1,& &J_-E_3 = -E_4,& &J_-E_4 = E_3.
	\end{aligned}
\end{equation}
None of these structures is K\"ahler, see Wall~\cite{Wall-LNM}.
However, both of them are  globally conformal K\"ahler, meaning that $\tilde g$ is globally conformal to a K\"ahler metric, namely~$e^{2t} \tilde g$.
This can be checked by verifying that the two-forms $e^{2t} \Omega_+$ and $e^{2t} \Omega_-$ are closed, where $\Omega_\pm(\cdot\,,\cdot)=\tilde g(\cdot\,,J_\pm \,\cdot)$.
It is also worth mentioning that, as a complex manifold, $\mathrm{Sol}_{0}^{4}$ is holomorphically isometric to the universal covering of an Inoue surface \cite{Inoue} equipped with Tricerri metric \cite{Tricerri, V87}.

Since the frame $\{E_1,E_2,E_3,E_4\}$ from \eqref{eq:onf} is orthonormal with Lie brackets \eqref{eq:Lie_brackets}, Koszul's formula can be used to compute the Levi-Civita connection $\tilde \nabla$ of $\tilde g$:
\begin{equation}
\label{eq:levicivitaconnection}
	\arraycolsep=10pt\def\arraystretch{1.2}
  \begin{array}{llll}
    \tilde \nabla_{E_1}E_1=E_4,
     & \tilde \nabla_{E_1}E_2=0,
     & \tilde \nabla_{E_1}E_3=0,
     & \tilde \nabla_{E_1}E_4=-E_1 ,  \\
    \tilde \nabla_{E_2}E_1=0,
     & \tilde \nabla_{E_2}E_2=E_4,
     & \tilde \nabla_{E_2}E_3=0,
     & \tilde \nabla_{E_2}E_4=-E_2 ,  \\
    \tilde \nabla_{E_3}E_1=0,
     & \tilde \nabla_{E_3}E_2=0,
     & \tilde \nabla_{E_3}E_3=-2E_4,
     & \tilde \nabla_{E_3}E_4=2E_3,   \\
    \tilde \nabla_{E_4}E_1=0,
     & \tilde \nabla_{E_4}E_2=0,
     & \tilde \nabla_{E_4}E_3=0,
     & \tilde \nabla_{E_4}E_4=0.
  \end{array}
\end{equation}
From the Levi-Civita connection one can compute the Riemann curvature tensor $\tilde R$ of $\sol$ to obtain the following non-zero components:
\begin{alignat*}{2}
  \tilde R(E_1,E_2)E_2 & = - E_1,
  &&\qquad \tilde R(E_2,E_3)E_3 = 2E_2 , \\
  \tilde R(E_1,E_3)E_3 & = 2 E_1,
  &&\qquad \tilde R(E_2,E_4)E_4 = -E_2, \\
  \tilde R(E_1,E_4)E_4 & = -E_1,
  &&\qquad \tilde R(E_3,E_4)E_4 = -4E_3.
\end{alignat*}
In particular, if $\tilde K_{ij}$ is the sectional curvature of the plane $\mathrm{span}\{E_i, E_j\}$, then
\begin{equation}
	\label{eq:sec-curv}
  \tilde K_{12} = -1,\quad \tilde K_{13} = 2,\quad \tilde K_{14} = -1,\quad \tilde K_{23} = 2,\quad \tilde K_{24} = -1,\quad \tilde K_{34} = -4.
\end{equation}

Using the complex structures $J_{+}$ and $J_{-}$ from \eqref{eq:complex-structures}, we have the following invariant expression for the curvature tensor of $\sol$:
\begin{equation} \label{eq:R}
\begin{aligned}
	\tilde R(X,Y)Z
	= \ & 2\, ( \tilde g(Y,Z) X - \tilde g(X,Z) Y ) \\
	& -\frac 12 \, ( \tilde g(J_+Y,Z)J_+X - \tilde g(J_+X,Z)J_+Y + 2 \, \tilde g(X,J_+Y)J_+Z ) \\
	& -\frac 12 \, ( \tilde g(J_-Y,Z)J_-X - \tilde g(J_-X,Z)J_-Y + 2 \, \tilde g(X,J_-Y)J_-Z ) \\
	& -3 \, ( \tilde g(PY,Z)X + \tilde g(Y,Z)PX - \tilde g(PX,Z)Y - \tilde g(X,Z)PY )
\end{aligned}
\end{equation}
for $X,Y,Z \in \mathfrak X(\sol)$, where $P$ denotes projection onto $E_4$, i.e., $P\,\cdot = \tilde g(\cdot\,,E_4)E_4$.

\begin{remark}
	\normalfont
It follows from \eqref{eq:R} that the tensors $J_+$, $J_-$ and $P$ play a fundamental role in the geometry of $\sol$. Since it may be useful for future research, we compute their covariant derivatives with respect to the Levi-Civita connection:
\begin{alignat*}{2}
	&(\tilde \nabla J_\pm)(X,Y)
	&&= -\tilde g(J_\pm Y,E_4)X + \tilde g(Y,E_4)J_\pm X + \tilde g(J_\pm Y,X) E_4 - \tilde g(Y,X)J_\pm E_4, \\
	&(\tilde \nabla P)(X,Y)
	&&= \frac 12 \, ( \tilde g(Y,X)E_4 + \tilde g(Y,E_4)X )
	- 4 \tilde g(Y,PX) E_4 \\
	& &&\quad
	+ \frac32 \, ( \tilde g(Y,J_+J_-X)E_4 + \tilde g(Y,E_4) J_+J_- X ).
\end{alignat*}
Note that the first formula shows again that $(\sol,\tilde g,J_{\pm})$ is not K\"ahler. To obtain this first formula, we used equation~(2.6) in~\cite{V76}. To obtain the second one, we noticed that $\tilde \nabla_X E_4 = \frac 12 \, X - 2 \, PX + \frac 32 \, J_+J_-X$ for any $X \in \mathfrak X(\sol)$.
\end{remark}

\section{Classification of Codazzi hypersurfaces of $\sol$}
\label{sec:codazzi}

In this section, we obtain a full classification of Codazzi hypersurfaces of $\sol$.
First, we describe the possible unit normal vector fields to such a hypersurface.

\begin{proposition}
\label{prop:codazzi}
Let $M$ be a Codazzi hypersurface of $\sol$.
Then a local unit normal vector field $N$ to $M$ takes one of the following forms with respect to the frame~\eqref{eq:onf}:
\begin{enumerate}[(a),wide=\parindent]
\item
$N = \pm E_3$,
\item
$N=\pm E_4$,
\item
$N = \cos\alpha\, E_1 + \sin\alpha\, E_2$ for some local smooth function $\alpha$ on $M$ satisfying
\[ E_3(\alpha) = E_4(\alpha) = 0. \]
\end{enumerate}

Conversely, there exist hypersurfaces of $\sol$ admitting a local unit normal vector field of any of the forms above.
\end{proposition}

\begin{proof}
	Let $M$ be a Codazzi hypersurface of $\sol$ and assume that its unit normal vector field takes the form  $N = aE_1 + bE_2 + cE_3 + dE_4$ for some local smooth functions $a$, $b$, $c$ and $d$ on $M$, satisfying $a^2+b^2+c^2+d^2=1$.
	Then the following three vector fields form a local orthonormal frame on $M$:
	\begin{equation}
	\label{eq:tangent-frame}
	\begin{aligned}
		& T_1 = bE_1  -aE_2 + dE_3 -cE_4, \\
		& T_2 = cE_1 -dE_2 -aE_3 + bE_4, \\
		& T_3 = dE_1 + cE_2  -bE_3 -aE_4.
	\end{aligned}
	\end{equation}
	Note that there exist local smooth functions $r$, $\alpha$ and $\beta$ on $M$ such that $a = r\cos\alpha$, $b = r\sin\alpha$, $c = \sqrt{1-r^2}\cos\beta$ and $d = \sqrt{1-r^2}\sin\beta$.

	The Codazzi equation for Codazzi hypersurfaces reduces to
	\begin{equation}
	\label{eq:codazzi-cod}
	\tilde g(\tilde R(X,Y)Z, \, N) = 0
	\end{equation}
	for $X,Y,Z \in \mathfrak X (M)$.
	In particular, taking $X=T_2$, $Y=T_3$ and $Z=T_1$ yields $6r^2(1-r^2) = 0$.
	Since $r$ is smooth, it follows that $r$ is locally constant, namely either $r = 0$ or $r = 1$.

	\emph{Case 1: $r=0$.} In this case, $N = \cos\beta\, E_3 + \sin\beta\, E_4$ and the orthogonal complement to $N$ is spanned by $E_1$, $E_2$ and $T_1=\sin\beta \, E_3 - \cos\beta \, E_4$.
	Plugging $X=T_1$, $Y=E_1$ and $Z=E_1$ in~\eqref{eq:codazzi-cod} yields $3\cos\beta\sin\beta = 0$.
	Since $\beta$ is smooth, this implies that $\beta$ is locally constant and $\cos\beta=0$ or $\sin\beta=0$.
	If $\sin\beta = 0$, then $N = \pm E_3$ and the distribution orthogonal to $N$ is spanned by $E_1$, $E_2$ and $E_4$, which is integrable by~\eqref{eq:Lie_brackets}.
	If $\cos\beta = 0$, then $N = \pm E_4$ and the distribution orthogonal to $N$ is spanned by $E_1$, $E_2$ and $E_3$, which is also integrable by~\eqref{eq:Lie_brackets}.
	The integrability of these distributions proves the converse.

	\emph{Case 2: $r=1$.} In this case, $N = \cos\alpha\, E_1 + \sin\alpha\, E_2$ and the orthogonal complement of $N$ is spanned by $E_3$, $E_4$ and $T_1=\sin\alpha \, E_1 - \cos\alpha \, E_2$.
	This distribution is integrable if and only if $E_3(\alpha) = E_4(\alpha) = 0$ by~\eqref{eq:Lie_brackets}.
	Note that this also proves the converse statement.
\end{proof}

We can now explicitly describe all Codazzi hypersurfaces of $\sol$.

\begin{theorem}
\label{thm:codazzi}
Let $M$ be a connected Codazzi hypersurface of $\sol$. Then one of the following holds.
\begin{enumerate}[(a),leftmargin=*]
\item $M$ is an open subset of $\set{M(x,y,z,t) \in \sol \mid z = c}$ for some $c \in \R$ and has constant sectional curvature $-1$.
\item $M$ is an open subset of $\set{M(x,y,z,t) \in \sol \mid t = c}$ for some $c \in \R$ and has constant sectional curvature $0$.
\item There exist local coordinates $(u_1,u_2,u_3)$ on $M$ such that the immersion of $M$ into $\sol$ is given by $F(u_1,u_2,u_3) =( \gamma_1(u_1),\gamma_2(u_1),u_2,u_3)$, where $\gamma=(\gamma_1,\gamma_2)$ is any smooth curve in the $xy$-plane.
\end{enumerate}

Conversely, all of the hypersurfaces listed above are Codazzi hypersurfaces.
\end{theorem}
\begin{proof}
Let $M$ be a connected Codazzi hypersurface and let $U$ be a connected open subset of $M$ on which a unit normal vector field $N$ to $M$ is defined. According to Proposition~\ref{prop:codazzi}, there are three possibilities for $N$, which we will treat separately.

\emph{Case 1: $N = \pm E_3$.} In this case, $E_1$, $E_2$ and $E_4$ span the tangent space to~$U$.
It follows from \eqref{eq:onf} that $U$ is an open subset of $\set{M(x,y,z,t) \in \sol \mid z = c}$ for some $c \in \R$.
Considering the possibilities for $N$ in Proposition~\ref{prop:codazzi} and taking into account that $M$ is connected and smooth, we conclude that $M$ is globally an open subset of $\set{M(x,y,z,t) \in \sol \mid z = c}$.
The equation of Gauss implies that $\tilde g(R(E_1,E_2)E_2,E_1) = \tilde g(R(E_1,E_4)E_4,E_1) = \tilde g(R(E_2,E_4)E_4,E_2) = -1$, showing that $M$ has constant sectional curvature $-1$.

\emph{Case 2: ${N = \pm E_4}$.} Similarly as in Case~1, we obtain that $M$ is an open subset of $\set{M(x,y,z,t) \in \sol \mid t = c}$ for some $c \in \R$ and that this hypersurface has constant sectional curvature $0$.

\emph{Case 3: $N = \cos\alpha \,E_1 + \sin\alpha \,E_2$ for some $\alpha \in C^\infty(U)$ satisfying $E_3(\alpha) = E_4(\alpha) = 0$}.
In this case, $\sin\alpha \, E_1 - \cos\alpha \, E_2$, $E_3$ and $E_4$ span the tangent space to $U$.

In order to obtain the immersion of $U \subseteq M$ into $\sol$ explicitly, we first look for local coordinate vector fields on $U$ of the form
\begin{equation}
\label{eq:codazzi-coordvfs}
\partial_{u_1} = \lambda(\sin\alpha \, E_1 - \cos\alpha \, E_2), \ \ \partial_{u_2} = \mu E_3, \ \  \partial_{u_3} = E_4
\end{equation}
for some local real-valued functions $\lambda$ and $\mu$ on $U$ without any zeroes.
Note that these vector fields are linearly independent by construction and hence they are coordinate vector fields if and only if their Lie brackets vanish.
A direct computation using the Lie brackets~\eqref{eq:Lie_brackets} and the fact that $E_3(\alpha)=E_4(\alpha)=0$ shows
\begin{align*}
& [\partial_{u_1},\partial_{u_2}] = -\frac{\lambda_{u_2}}{\lambda}\partial_{u_1} +\frac{\mu_{u_1}}{\mu}\partial_{u_2}, \\
& [\partial_{u_1},\partial_{u_3}] = -\left( 1+ \frac{\lambda_{u_3}}{\lambda}\right) \partial_{u_1}, \\
& [\partial_{u_2},\partial_{u_3}] = \left( 2- \frac{\mu_{u_3}}{\mu}\right) \partial_{u_2},
\end{align*}
which vanish if and only if $\lambda_{u_2}=0$, $\lambda_{u_3}=-\lambda$, $\mu_{u_1}=0$ and $\mu_{u_3}=2\mu$. We only need one coordinate system and choose
\begin{equation}
\label{eq:fg-codazzi}
\lambda(u_1,u_2,u_3) = e^{-u_3}, \ \ \mu(u_1,u_2,u_3) = e^{2u_3}.
\end{equation}
The condition $E_3(\alpha)=E_4(\alpha)=0$ is equivalent to $\alpha_{u_2}=\alpha_{u_3}=0$, which means that $\alpha$ is a function of $u_1$ only.

Now let $F(u_1,u_2,u_3) \!=\! ( F_1(u_1,u_2,u_3), F_2(u_1,u_2,u_3), F_3(u_1,u_2,u_3), F_4(u_1,u_2,u_3) )$ be a local parametrization of $M$ in $\sol$ using the coordinate system constructed above. From \eqref{eq:codazzi-coordvfs}, in combination with \eqref{eq:onf} and \eqref{eq:fg-codazzi}, we obtain
\begin{equation}
	\label{eq:diffeq-codazzi}
\begin{aligned}
& ((F_1)_{u_1}, (F_2)_{u_1}, (F_3)_{u_1}, (F_4)_{u_1}) =
( e^{F_4-u_3}\sin\alpha(u_1), \, -e^{F_4-u_3}\cos\alpha(u_1), \, 0, \, 0 ), \\
& ((F_1)_{u_2}, (F_2)_{u_2}, (F_3)_{u_2}, (F_4)_{u_2}) =
( 0, \, 0, \, e^{-2(F_4-u_3)} , \, 0 ), \\
& ((F_1)_{u_3}, (F_2)_{u_3}, (F_3)_{u_3}, (F_4)_{u_3}) =
( 0, \, 0, \, 0, \, 1).
\end{aligned}
\end{equation}
From the fourth components of these equations, it follows that
\begin{equation*}
F_4(u_1,u_2,u_3) = u_3 + c_4
\end{equation*}
for some $c_4 \in \R$. After plugging this into the other components of~\eqref{eq:diffeq-codazzi}, we find
\begin{align*}
& F_1(u_1,u_2,u_3) = e^{c_4} \! \int \! \sin\alpha(u_1)\,du_1 + c_1, \\
& F_2(u_1,u_2,u_3) = -e^{c_4} \! \int \! \cos\alpha(u_1)\,du_1 + c_2, \\
& F_3(u_1,u_2,u_3) = e^{-2c_4}u_2 + c_3
\end{align*}
for some $c_1,c_2,c_3 \in \R$.
After a left translation by $(-c_1e^{-c_4},-c_2e^{-c_4},-c_3e^{2c_4},-c_4)$, which is an isometry of $\sol$, we obtain
\begin{equation*}
F(u_1,u_2,u_3) = \left( \int \! \sin\alpha(u_1)\,du_1, \, - \! \int \! \cos\alpha(u_1)\,du_1, \, u_2, \, u_3 \right),
\end{equation*}
see \eqref{eq:groupoperation}.
The parametrization in the statement of the theorem is obtained by noting that the first two components parametrize an arbitrary smooth curve in $\set{M(x,y,z,t) \in \sol \mid z = z_0, \, t = t_0}$, which is a Euclidean plane for any $z_0,t_0 \in \R$ by \eqref{eq:metric}.

Now we prove the converse of the statement.

\emph{Case (a).} In this case the vector fields $E_1$, $E_2$ and $E_4$ form an orthonormal frame to $M$ and $E_3$ is a unit normal vector field.
A straightforward computation using~\eqref{eq:gauss-formula} and the connection~\eqref{eq:levicivitaconnection} shows that the second fundamental form vanishes.
Hence $M$ is totally geodesic and thus also a Codazzi hypersurface.

\emph{Case (b).} In this case the vector fields $E_1$, $E_2$ and $E_3$ form an orthonormal frame to $M$ and $E_4$ is a unit normal vector field.
Using~\eqref{eq:gauss-formula} and~\eqref{eq:levicivitaconnection}, one can compute that the second fundamental form is given by
\begin{equation}
	\label{eq:2ndff-codazzi-1}
	\begin{aligned}
		h(E_1,E_2) &= h(E_1,E_3) = h(E_2,E_3) = 0, \\
		h(E_1,E_1) &= h(E_2,E_2) = E_4, \\
		h(E_3,E_3) &= -2E_4.
	\end{aligned}
\end{equation}
From~\eqref{eq:levicivitaconnection} we see that the induced Levi-Civita connection $\nabla$ on $M$ vanishes and $\tilde\nabla_X E_4$ is tangent to $M$ for $X \in \set{E_1,E_2,E_3}$.
Together with~\eqref{eq:2ndff-codazzi-1} this implies that $\nabla h = 0$, i.e.\ $M$ is parallel and thus also a Codazzi hypersurface.

\emph{Case (c).} Let $F$ be the isometric immersion given in the statement of the theorem.
Then $E_3$, $E_4$ and $W :=
\frac{\gamma_1^\prime E_1 + \gamma_2^\prime E_2}{\sqrt{(\gamma_1^\prime)^2+(\gamma_2^\prime)^2}}$ form a local orthonormal frame for this hypersurface.
A unit normal vector field is given by $N = \frac{\gamma_2^\prime E_1 - \gamma_1^\prime E_2}{\sqrt{(\gamma_1^\prime)^2+(\gamma_2^\prime)^2}}$.
From~\eqref{eq:metric} and~\eqref{eq:levicivitaconnection} it follows that the second fundamental form associated to the immersion is given by
\begin{equation}
	\label{eq:2ndff-codazzi-2}
	\begin{aligned}
		&h(W,W) =
		\frac{e^{w}(\gamma_1^{\prime\prime}\gamma_2^\prime - \gamma_1^{\prime}\gamma_2^{\prime\prime})}{((\gamma_1^\prime)^2+(\gamma_2^\prime)^2)^{3/2}} \; N = e^{w}\kappa_\gamma \, N, \\
		&h(E_3,E_3) = h(E_3,E_4) = h(E_3,W) = h(E_4,E_4) = h(E_4,W) = 0,
	\end{aligned}
\end{equation}
where $\kappa_\gamma$ is the curvature of $\gamma=(\gamma_1,\gamma_2)$ in the Euclidean plane.
Using~\eqref{eq:gauss-formula}, \eqref{eq:weingarten-formula} and the connection~\eqref{eq:levicivitaconnection} it is now a straightforward computation to show that $\nabla h$ is symmetric.
Hence the immersion $F$ is Codazzi.
\end{proof}

The class of Codazzi hypersurfaces includes all parallel hypersurfaces and, in particular, the totally geodesic ones.
\begin{corollary}
\label{cor:parallel}
Let $M$ be a connected parallel hypersurface of $\sol$. Then one of the following holds.
\begin{enumerate}[(a),leftmargin=*]
\item $M$ is an open subset of $\set{M(x,y,z,t) \in \sol \mid z = c}$ for some $c \in \R$ and has constant sectional curvature $-1$.
\item $M$ is an open subset of $\set{M(x,y,z,t) \in \sol \mid t = c}$ for some $c \in \R$ and has constant sectional curvature $0$.
\item $M$ is an open subset of $\set{M(x,y,z,t) \in \sol \mid a x + b y = c}$ for some $a,b,c \in \R$ and does not have constant sectional curvature.
\end{enumerate}

Conversely, all of the hypersurfaces listed above are parallel.
\end{corollary}

\begin{proof}
Cases (a) and (b) follow directly from the proof of Theorem~\ref{thm:codazzi}.
We obtain case (c) as a special case of (c) in Theorem~\ref{thm:codazzi} by noticing that $\nabla h$ is symmetric if and only if the curvature $\kappa_\gamma$ of $\gamma$ vanishes.
In that case, $\gamma$ is a straight line and $h$ vanishes identically by~\eqref{eq:2ndff-codazzi-2}.
One computes the sectional curvatures from~\eqref{eq:sec-curv}, using that $E_3$, $E_4$ and $bE_1-aE_2$ form a tangent frame for this hypersurface.
\end{proof}

\begin{corollary}
\label{cor:tot-geod}
Let $M$ be a connected totally geodesic hypersurface of $\sol$. Then one of the following holds.
\begin{enumerate}[(a),leftmargin=*]
\item $M$ is an open subset of $\set{M(x,y,z,t) \in \sol \mid z = c}$ for some $c \in \R$ and has constant sectional curvature $-1$.
\item $M$ is an open subset of $\set{M(x,y,z,t) \in \sol \mid a x + b y = c}$ for some $a,b,c \in \R$ and does not have constant sectional curvature.
\end{enumerate}

Conversely, all of the hypersurfaces listed above are totally geodesic.
\end{corollary}

\begin{proof}
This follows from the proof of Corollary~\ref{cor:parallel}.
\end{proof}

\begin{remark}
	\normalfont
For each of the types of hypersurfaces listed in Corollary~\ref{cor:parallel} and Corollary~\ref{cor:tot-geod}, all hypersurfaces of that type are congruent via isometries of $\sol$.
\end{remark}

\begin{remark}
	{\normalfont
The totally geodesic hypersurfaces in Corollary~\ref{cor:tot-geod} are factors of the warped product representations of $\sol$ mentioned in~\cite{j-traj}, and this agrees with Theorem~1 in~\cite{nikolayevsky}.
Namely, we have}
\begin{align*}
	\sol &\cong \H^3(-1) \times_{e^{2t}} \R \cong \set{M(x,y,z,t) \in \sol \mid z = c} \times_{e^{2t}} \R, \\
	\sol &\cong \H^2(-4) \times_{e^{-t}} \R^2 \cong (\H^2(-4) \times_{e^{-t}} \R) \times_{e^{-t}} \R \\
	&\cong \set{M(x,y,z,t) \in \sol \mid ax + by = c} \times_{e^{-t}} \R. \qedhere
\end{align*}
\end{remark}

\section{Classification of totally umbilical hypersurfaces of $\sol$}
\label{sec:tot-umb}

In this section, we give a complete classification of totally umbilical hypersurfaces of $\sol$.
As in the previous section for Codazzi hypersurfaces, we start by listing the possible unit normal vector fields to such a hypersurface.

\begin{proposition}
\label{prop:codazzi-tot-umb}
Let $M$ be a totally umbilical hypersurface of $\sol$ with local unit normal vector field $N$. Let $H$ be the mean curvature vector field and $\lambda = \tilde g(H,N)$.
Then one of the following holds.
\begin{enumerate}[(a),leftmargin=*]
\item
$\lambda = 0$ and $N = \cos\alpha \, E_1 + \sin\alpha \, E_2$ for some $\alpha \in \R$.
\item
$\lambda = \sin\beta$ and $N = \cos\beta \, E_3 + \sin\beta \, E_4$ for some local smooth function $\beta$ on $M$ such that $\cos\beta$ is nowhere vanishing and
\[
E_1(\beta) = E_2(\beta)  = 0, \ \  (\sin\beta\,E_3-\cos\beta\,E_4)(\beta) = 3\sin\beta.
\]
\end{enumerate}

Conversely, there exist hypersurfaces of $\sol$ admitting a local unit normal vector field of any of the forms above.
\end{proposition}

\begin{proof}
Let $M$ be a totally umbilical hypersurface of $\sol$.
We define its unit normal vector field $N$ and tangent vector fields $T_1$, $T_2$ and $T_3$ as in the proof of Proposition~\ref{prop:codazzi}.
The Codazzi equation for totally umbilical hypersurfaces is
\begin{equation}
\label{eq:codazzi-tu}
\tilde g(\tilde R(X,Y)Z, \, N) = g(Y,Z)\, X(\lambda) - g(X,Z)\, Y(\lambda)
\end{equation}
for $X,Y,Z \in \mathfrak X (M)$.
However, when taking $X=T_2$, $Y=T_3$ and $Z=T_1$ the equation reduces again to $\tilde g(\tilde R(T_2,T_3)T_1),\, N) = 0$, which is equivalent to $6r^2(1-r^2) = 0$.
Since $r$ is smooth, it follows that $r$ is locally constant, namely either $r = 1$ or $r = 0$.

\emph{Case 1: $r=1$.} In this case, $N = \cos\alpha\, E_1 + \sin\alpha\, E_2$ and the orthogonal complement of $N$ is spanned by $E_3$, $E_4$ and $T_1=\sin\alpha \, E_1 - \cos\alpha \, E_2$.
This distribution is integrable if and only if $E_3(\alpha) = E_4(\alpha) = 0 $.
To proceed, we consider the Codazzi equation~\eqref{eq:codazzi-tu} for the tangent vector fields $T_1, E_3$ and $E_4$:
\begin{alignat*}{2}
&\tilde g(\tilde R(T_1,E_3)E_3,N)
&&= T_1(\lambda), \\
&\tilde g(\tilde R(E_3,E_4)E_4,N)
&&= E_3(\lambda), \\
&\tilde g(\tilde R(E_4,E_3)E_3,N)
&&= E_4(\lambda).
\end{alignat*}
A straightforward computation shows that the left-hand sides of these equations are zero.
Since $\set{T_1,E_3,E_4}$ is a frame on $M$, this implies that $\lambda$ is constant.

Because $M$ is totally umbilical, we have $h(X,X) = \lambda N$ for any $X \in \set{T_1,E_3,E_4}$.
From this condition we obtain
$\lambda N = T_1(\alpha)\,N$ and $\lambda N = 0$.
Therefore $\lambda = 0$ and $T_1(\alpha) = 0$.
Together with $E_3(\alpha) = E_4(\alpha) = 0 $ this implies that $\alpha$ is constant.

\emph{Case 2: $r=0$.} In this case, $N = \cos\beta\,E_3+\sin\beta\,E_4$ and the orthogonal complement of $N$ is spanned by $E_1$, $E_2$ and $T_1= \sin\beta\,E_3-\cos\beta\,E_4$.
This distribution is integrable if and only if $E_1(\beta) = E_2(\beta) = 0$.

Note that $\beta$ has to be such that $\cos\beta$ is never zero. That is because if $\cos\beta(p) = 0$ at some point $p \in M$, that point cannot be umbilical. Indeed, we have $N(p) = E_4(p)$ and $\set{E_1(p),E_2(p),E_3(p)}$ is an orthonormal basis for $T_pM$. From \eqref{eq:levicivitaconnection}, one obtains
$h(E_1(p),E_1(p)) = E_4(p)$, $h(E_2(p),E_2(p)) = E_4(p)$ and  $h(E_3(p),E_3(p)) = -2E_4(p)$.

We consider again the Codazzi equation~\eqref{eq:codazzi-tu} for some vector fields:
\begin{alignat*}{1}
&\tilde g(\tilde R(T_1,E_1)E_1,N)
= T_1(\lambda), \\
&\tilde g(\tilde R(E_1,E_2)E_2,N)
= E_1(\lambda), \\
&\tilde g(\tilde R(E_2,E_1)E_1,N)
= -E_2(\lambda).
\end{alignat*}
A straightforward computation shows that these equations are equivalent to
\begin{equation}
\label{eq:Tlambda}
T_1(\lambda) = 3\cos\beta\sin\beta, \ \ E_1(\lambda) = 0, \ \ E_2(\lambda) = 0.
\end{equation}
Since $M$ is totally umbilical, we again have $h(X,X) = \lambda N$ for any $X \in \set{T_1,E_1,E_2}$ and obtain
$\lambda N = (T_1(\beta) -2\sin\beta) N$ and $\lambda N = \sin\beta \, N$.
Hence $\lambda = \sin\beta$.
Combining this with \eqref{eq:Tlambda} we find $3\cos\beta\sin\beta = \cos\beta \,T_1(\beta)$.
From $\cos\beta \ne 0$ it follows that $T_1(\beta) = 3\sin\beta$.

For the converse, it suffices to check the integrability of the orthogonal complements of $N$ in both cases.
\end{proof}

\begin{theorem}
\label{thm:totally-umbilical}
Let $M$ be a connected totally umbilical hypersurface of $\sol$. Then one of the following holds.
\begin{enumerate}[(a),leftmargin=*]
	\item $M$ is an open subset of $\set{M(x,y,z,t) \in\sol \mid x\cos\alpha + y \sin\alpha = c}$ for some $c,\alpha \in \R$ and does not have constant sectional curvature.
	\item There exist local coordinates $(u_1,u_2,u_3)$ on $M$ such that the immersion of $M$ into $\sol$ is given by $F(u_1,u_2,u_3) =( u_1,u_2,\gamma_1(u_3),\gamma_2(u_3) )$, where $\gamma=(\gamma_1,\gamma_2)$ is any smooth curve in the $zt$-plane satisfying
	\begin{equation}
		\label{eq:diffeq-totumb}
		\gamma_1^{\prime\prime}\gamma_2^\prime - \gamma_1^\prime\gamma_2^{\prime\prime} + 5 \gamma_1^\prime (\gamma_2^\prime)^2 + 3 e^{4\gamma_2}(\gamma_1^\prime)^3 = 0.
	\end{equation}
\end{enumerate}
Conversely, all of the hypersurfaces listed above are totally umbilical.
\end{theorem}
\begin{proof}
	Let $M$ be a connected totally umbilical hypersurface and let $U$ be a connected open subset of $M$ on which a unit normal vector field $N$ to $M$ is defined.
	According to Proposition~\ref{prop:codazzi-tot-umb}, there are two possibilities for $N$, which we will treat separately.

	\emph{Case 1: $N = \cos\alpha\,E_1 + \sin\alpha\,E_2$ for some $\alpha \in \R$}.
	We know from Section~\ref{sec:codazzi} that $M$ is totally geodesic in this case and $N$ can be defined globally.
	Then $M$ is an open subset of $\set{M(x,y,z,t) \in \sol \mid a x + by = c}$ for some $a,b,c \in \R$.
	The sectional curvature is computed in the proof of Corollary~\ref{cor:parallel} and is not constant.

	\emph{Case 2: $N = \cos\beta \,E_3 + \sin\beta \,E_4$ for some $\beta \in C^\infty(U)$ satisfying ${E_1(\beta) = E_2(\beta) = 0}$, $(\sin\beta\, E_3 - \cos\beta\, E_4)(\beta) = 3\sin\beta$ and $\cos\beta \ne 0$.}
	In this case, $\sin\beta \, E_3 - \cos\beta \, E_4$, $E_1$ and $E_2$ span the tangent space to $U$.

	In order to obtain the immersion of $U \subseteq M$ into $\sol$ explicitly, we first look for local coordinate vector fields on $U$ of the form
	\begin{equation}
	\label{eq:totumb-coordvfs}
	\partial_{u_1} = \lambda E_1, \ \ \partial_{u_2} = \mu E_2, \ \  \partial_{u_3} = \sin\beta\, E_3 - \cos\beta\, E_4
	\end{equation}
	for some local real-valued functions $\lambda$ and $\mu$ on $U$ without any zeroes.
	Note that these vector fields are linearly independent by construction and hence they are coordinate vector fields if and only if their Lie brackets vanish.
	A direct computation using~\eqref{eq:gauss-formula} and the fact that $E_1(\beta)=E_2(\beta)=0$ shows
	\begin{align*}
	& [\partial_{u_1},\partial_{u_2}] = -\frac{\lambda_{u_2}}{\lambda}\partial_{u_1} +\frac{\mu_{u_1}}{\mu}\partial_{u_2}, \\
	& [\partial_{u_1},\partial_{u_3}] = \left( \cos\beta - \frac{\lambda_{u_3}}{\lambda}\right) \partial_{u_1}, \\
	& [\partial_{u_2},\partial_{u_3}] = \left( \cos\beta - \frac{\mu_{u_3}}{\mu}\right) \partial_{u_2},
	\end{align*}
	which vanish if and only if $\lambda_{u_2}=0$, $\lambda_{u_3}=\lambda\cos\beta$, $\mu_{u_1}=0$ and $\mu_{u_3}=\mu\cos\beta$.
	We only need one coordinate system and choose
	\begin{equation}
	\label{eq:fg-totumb}
	\lambda(u_1,u_2,u_3) = \mu(u_1,u_2,u_3) = \exp\left( \int \cos\beta(u_3)\, du_3 \right).
	\end{equation}
	Notice that the condition $E_1(\beta)=E_2(\beta)=0$ is equivalent to $\beta_{u_1}=\beta_{u_2}=0$, which means that $\beta$ is a function of $u_3$ only.
	From $(\sin\beta\, E_3 - \cos\beta-, E_4)(\beta) = 3 \sin\beta$ it follows that $\beta_{u_3}=3\sin\beta$.

	Now let $F(u_1,u_2,u_3) \!=\! ( F_1(u_1,u_2,u_3), F_2(u_1,u_2,u_3), F_3(u_1,u_2,u_3), F_4(u_1,u_2,u_3) )$ be a local parametrization of $M$ in $\sol$ using the coordinate system constructed above.
	From \eqref{eq:totumb-coordvfs}, in combination with \eqref{eq:onf} and \eqref{eq:fg-totumb}, we obtain
	\begin{equation}
		\label{eq:F-totumb}
		\begin{aligned}
			& ((F_1)_{u_1}, (F_2)_{u_1}, (F_3)_{u_1}, (F_4)_{u_1}) =
			( e^{F_4}\lambda(u_3), \, 0, \, 0, \, 0 ), \\
			& ((F_1)_{u_2}, (F_2)_{u_2}, (F_3)_{u_2}, (F_4)_{u_2}) =
			( 0, \, e^{F_4}\lambda(u_3), \, 0 , \, 0 ), \\
			& ((F_1)_{u_3}, (F_2)_{u_3}, (F_3)_{u_3}, (F_4)_{u_3}) =
			( 0, \, 0, \, e^{-2 F_4}\sin\beta(u_3), \, -\cos\beta(u_3)).
		\end{aligned}
	\end{equation}
	From the third and fourth components of these equations, it follows that
	\begin{equation*}
		\begin{aligned}
			&F_4(u_1,u_2,u_3) = -\int\cos\beta(u_3)\, du_3 + c_4 \\
			&F_3(u_1,u_2,u_3) = \int e^{-2F_4(u_3)}\sin\beta(u_3)\, du_3 + c_3
		\end{aligned}
	\end{equation*}
	for some $c_4,c_3 \in \R$.
	After plugging this into the first two components of~\eqref{eq:F-totumb}, using \eqref{eq:fg-totumb} we~find
	\begin{align*}
	& F_1(u_1,u_2,u_3) = e^{c_4} u_1 + c_1, \\
	& F_2(u_1,u_2,u_3) = e^{c_4} u_2 + c_2
	\end{align*}
	for some $c_1,c_2 \in \R$.
	Now we apply to the components of $F$ a left translation by $(-c_1e^{-c_4},-c_2e^{-c_4},-c_3e^{2c_4},-c_4)$ (see~\eqref{eq:groupoperation}), which is an isometry of $\sol$.
	The resulting parametrization is
	\[
	\left( u_1,u_2,
	\int \exp\left(2\int\cos\beta(u_3)\, du_3\right) \sin\beta(u_3)\, du_3,
	-\int\cos\beta(u_3)\, du_3 \right).
	\]
	Denote by $\gamma_1(u_3)$ and $\gamma_2(u_3)$ the third and fourth component respectively.
	Using $\beta'=3\sin\beta$,
	 we find $\gamma_1^\prime = e^{-2\gamma_2}\sin\beta$, $\gamma_1^{\prime\prime}=5e^{-2\gamma_2}\cos\beta\sin\beta$ and $\gamma_2^\prime = -\cos\beta$, $\gamma_2^{\prime\prime}=3\sin^2\beta$.
   Then it is straightforward to check that the differential equation~\eqref{eq:diffeq-totumb} is satisfied.

	Now we prove the converse statement.

	\emph{Case~(a).}
	This follows from Corollary~\ref{cor:tot-geod} where we show that a hypersurface of this type is totally geodesic, hence totally umbilical.

	\emph{Case~(b).}
	Let $F$ be the isometric immersion given in the statement of the theorem.
	Then $E_1,E_2$ and $W := \frac{\gamma_1^\prime e^{2\gamma_2}E_3 + \gamma_2^\prime E_4}{\sqrt{(\gamma_1^\prime)^2e^{4\gamma_2}+(\gamma_2^\prime)^2}}$ form a local orthonormal frame for this hypersurface.
	A unit normal vector field is given by $N = \frac{\gamma_2^\prime E_3 - \gamma_1^\prime e^{2\gamma_2} E_4}{\sqrt{(\gamma_1^\prime)^2e^{4\gamma_2}+(\gamma_2^\prime)^2}}$.
	From~\eqref{eq:metric}, \eqref{eq:levicivitaconnection} and the differential equation~\eqref{eq:diffeq-totumb} it follows that the second fundamental form associated to the immersion is given by
	\begin{equation}
		\label{eq:2ndff-totumb}
		\begin{aligned}
			&h(E_1,E_2) = h(E_1,W) = h(E_2,W) = 0, \\
			&h(W,W) = h(E_1,E_1) = h(E_2,E_2) \\
			&\hspace*{3.9em} =\frac{e^{2\gamma_2}(\gamma_1^{\prime\prime}\gamma_2^\prime - \gamma_1^\prime\gamma_2^{\prime\prime} + 4\gamma_1^\prime(\gamma_2^\prime)^2 + 2e^{4\gamma_2}(\gamma_1^\prime)^3)}{((\gamma_1^\prime)^2e^{4\gamma_2}+(\gamma_2^\prime)^2)^{3/2}} \, N.
		\end{aligned}
	\end{equation}
	The associated mean curvature vector field is given by
	\begin{equation}
		\label{eq:meancurv-totumb}
		H = \frac{e^{2\gamma_2}(\gamma_1^{\prime\prime}\gamma_2^\prime - \gamma_1^\prime\gamma_2^{\prime\prime} + 4\gamma_1^\prime(\gamma_2^\prime)^2 + 2e^{4\gamma_2}(\gamma_1^\prime)^3)}{((\gamma_1^\prime)^2e^{4\gamma_2}+(\gamma_2^\prime)^2)^{3/2}} \, N.
	\end{equation}
	From equations~\eqref{eq:2ndff-totumb}, \eqref{eq:meancurv-totumb} and the fact that $\set{E_1,E_2,W}$ is an orthonormal frame, it follows that $h(\cdot,\cdot) = g(\cdot,\cdot) H$ hence the immersion $F$ is totally umbilical.
\end{proof}

\begin{remark}
	\normalfont
Case~(b) in Theorem~\ref{thm:totally-umbilical} includes the totally geodesic hypersurface $\set{M(x,y,z,t) \in \sol \mid z = c}$ where $c \in \R$ (see Corollary~\ref{cor:tot-geod}).
Indeed, the calculation of the second fundamental form associated to the immersion $F$, equation~\eqref{eq:2ndff-totumb}, shows that $h=0$ if and only if $\gamma_1^\prime = 0$ which is equivalent to the $z$-coordinate of the immersion being constant.
\end{remark}

\bibliographystyle{plain}
\bibliography{main}

\end{document}